\documentclass[10pt]{amsart}
\usepackage{amscd}      % to be able to use "CD" environment
\usepackage{amssymb}
\usepackage{amsmath}
\usepackage{xypic}      % to be able to use "diagram" environment of
\LaTeXdiagrams          % xypic to typeset commutative diagrams
\usepackage[all,v2]{xy}
\usepackage{dsfont}
\xyoption{2cell} \UseAllTwocells \xyoption{frame} \CompileMatrices
\allowdisplaybreaks[3]
\usepackage{hyperref}
\usepackage{xcolor}

\usepackage{latexsym}
\usepackage{epsfig}
\usepackage{amsfonts}
\usepackage{enumerate}

\newtheorem{prop}{Proposition}[section]
\newtheorem{lem}[prop]{Lemma}

\newtheorem{cor}[prop]{Corollary}
\newtheorem{thm}[prop]{Theorem}

\numberwithin{equation}{section}
\theoremstyle{definition}
\newtheorem{defn}[prop]{Definition}

\newtheorem{example}[prop]{Example}

\newtheorem{rmk}[prop]{Remark}

            {\nolinebreak $\Box$ \end{trivlist}}

\newcommand{\noprint}[1]{}

\newcommand{\Ext}{\mbox{Ext}}

\newcommand{\Hom}{\mbox{Hom}}

\newcommand{\XX}{{\mathfrak X}}

\newcommand{\YY}{{\mathfrak Y}}

\newcommand{\zz}{{\mathbb Z}}

\newcommand{\aaa}{{\mathbb A}}

\newcommand{\nn}{{\mathbb N}}

\newcommand{\m}{{\mathfrak m}}
\renewcommand{\ll}{{\mathbb L}}
\newcommand{\qq}{{\mathbb Q}}
\newcommand{\pp}{{\mathbb P}}
\newcommand{\cc}{{\mathbb C}}

\newcommand{\sT}{{\mathcal T}}

\newcommand{\sE}{{\mathcal E}}

\newcommand{\sO}{{\mathcal O}}

\newcommand{\sX}{{\mathcal X}}

\newcommand{\sM}{{\mathcal M}}

\DeclareMathOperator{\Sch}{Sch}

\DeclareMathOperator{\ind}{ind}
\DeclareMathOperator{\Ob}{Ob}

\DeclareMathOperator{\QG}{QG}

\DeclareMathOperator{\Der}{Der}

\newcommand{\rk}{\mathop{\rm rk}}

\renewcommand{\Im}{\mathop{\rm Im}}

\newcommand{\cok}{\mathop{\rm cok}}

\newcommand{\spec}{\mathop{\rm Spec}\nolimits}

\numberwithin{equation}{subsection}

\newcommand {\mat}      [1] {\left(\begin{array}{#1}}
\newcommand {\rix}          {\end{array}\right)}

\title[Higher obstruction space for surface singularities]{A note on higher obstruction space for surface singularities}

\author{Yunfeng Jiang}
\address{Department of Mathematics\\ University of Kansas\\ 405 Snow Hall 1460 Jayhawk Blvd\\Lawrence\\ KS 66045\\ USA}
\email{y.jiang@ku.edu}

\date{\today}

\begin{document}
\begin{abstract}
In this paper we collect some results on the obstruction spaces for rational surface singularities and minimally elliptic surface singularities. Based on the known results we  calculate higher obstruction spaces  for such surface singularities. 
The results  imply that in general the higher obstruction spaces of deforming semi-log-canonical surfaces do not vanish.  We apply the calculation result to show that there is no Li-Tian and Behrend-Fantechi  style virtual fundamental class on such moduli space of semi-log-canonical surfaces. 
\end{abstract}

%%% ----------------------------------------------------------------------
\maketitle
%%% ----------------------------------------------------------------------

\section{Introduction}

 In this short note we collect and calculate some results on  the higher obstruction spaces for  surface singularities.   
We basically follow the method in \cite{BC}, \cite{BC2} where the authors calculated the obstruction space $\sT_X^2$ (second tangent space) for  rational surface singularities and minimally elliptic surface singularities. 
The non-vanishing of $\sT^2_X$ implies that the singularity is not l.c.i. 
The higher obstruction spaces $\sT^i_X$ for $i\ge 3$ can help us to determine if  there exists a  perfect obstruction theory in the sense of Li-Tian, Behrend-Fantechi on the  moduli stack of s.l.c. surfaces, see \cite{Jiang_2021}. 

 \subsection{Rational surface singularities}\label{subsec_rational_surface}

Let $(X,p)$ be a normal  surface singularity.  It is called $rational$ if its minimal resolution $\pi: \widetilde{X}\to X$ satisfies the property that 
$R^1\pi_*(\sO_{\widetilde{X}})=0$. Let $E=\cup_{i=1}^{n} E_i$ be the exceptional locus of $\pi$.  Artin \cite{Artin1} showed that the surface singularity $X$ is rational if and only if for each divisor $D=\sum_{i=1}^n a_iE_i$ with 
$a_i\ge 0$ supported on the exceptional locus the arithmetic genus is negative.   Thus each irreducible component $E_i\cong \pp^1$. 

Let $Z$ be the fundamental cycle for $\pi$, i.e., the minimal cycle supported on the exceptional locus 
such that $Z>\sum_{i=1}^{n}E_i$ and $Z\cdot E_i\leq 0$ for all $i=1,2, \cdots, n$. 
$Z$ has the property that its arithmetic genus is exactly zero.  Cyclic quotient surface  singularities are examples of rational surface singularities. 

\subsection{Minimally elliptic surface singularities}\label{subsec_elliptic_surface}

Recall from \cite[Theorem 3.4]{Laufer}, for the minimal resolution  $\pi: \widetilde{X}\to X$, the following conditions are equivalent for the fundamental cycle 
$Z$:
\begin{enumerate}
\item $Z$ is a minimally elliptic cycle, i.e., $\chi(Z)=0$ and 
$\chi(D)>0$ for all cycles $D$ such that $0< D <Z$. 
\item $E_i\cdot Z=-E_i\cdot K_X$ for all irreducible components $E_i$ in $E=\pi^{-1}(p)$.
\item  $\chi(Z)=0$ and any connected proper subvariety of $E$ is the exceptional set for a rational singularity. 
\end{enumerate}
The normal surface singularity $(X,p)$ is called $minimally~elliptic$ if its minimal resolution $\pi: \widetilde{X}\to X$ satisfies the equivalent conditions above. 
Examples of minimally elliptic surface singularities contain simple elliptic singularities and cusp singularities which are known singularities for semi-log-canonical surfaces.

\subsection{Main results}\label{subsec_main_results}

Let us first introduce the general hypersurface section for a surface singularity $(X,0)$. Let $\sO_X=\sO_{X,0}$ be the local ring of $X$ at $0$, and 
$f\in \sO_X$.  From \cite{BC}, $Y=f^{-1}(0)$ is a hypersurface section on $X$ if the element 
$f\in \m_{X}$, the maximal ideal of the germ at $0$,  and $f$ is a non-zero divisor in $\sO_X$. 
Let us suppose $(Y,0)$ is an isolated singularity. 

For the singularity germ $(X,0)$, let $\sT_X^i$ be the higher tangent spaces for $i\geq 0$.  Note that $\sT_X^0$ is the tangent sheaf of the germ singularity, $\sT_X^1$ is the tangent space and 
$\sT_X^2$ is the local obstruction space.  The global tangent spaces $T_X^i$ is determined by $\sT_X^i$ by the local to global spectral sequence.

\begin{thm}\label{thm_rational_surface_intro}(Theorem \ref{thm_rational_surface})
Let $(X,0)$ be a rational surface singularity of embedding dimension $n+1\ge 4$, and 
$f$ is a general hypersurface section. 
Let $\pi: \widetilde{X}\to X$ be the minimal resolution.  Let 
$E=\cup E_i$ be the exceptional divisor and $Z$ the fundamental cycle. Then 
\begin{enumerate}
\item $\dim_{\cc}(\sT_X^2/\m \sT_X^2)=\dim_{\cc}(\sT_X^2/f \sT_X^2)=(n-1)(n-3)$, where $\m\subset \sT_X^2$ is the maximal ideal of 
$\sT_X^2$ taken as a module over the local ring $\sO_{X,0}$. 
\item If the fundamental cycle $Z$ is reduced and if for any connected subgraph $E^\prime\subset E$ with $Z\cdot E^\prime=0$ the self-intersection number
$E^\prime\cdot E^\prime=-2$ or $-3$, then we have $\dim_{\cc}(\sT_X^2)=(n-1)(n-3)$.
\item If $X$ is a quotient singularity then $\dim_{\cc}(\sT_X^2)=(n-1)(n-3)$.
\item If $X$ is a quotient singularity then $\dim_{\cc}(\sT_X^3)=\frac{1}{2}(n-1)(n-2)(n-3)$.
\item If $X$ is a quotient singularity then $\dim_{\cc}(\sT_X^4)=\frac{1}{6}(n-1)(n-2)(2n^2-8n+9)$.
\end{enumerate}
\end{thm}

\begin{rmk}
It is worth mentioning that using the techniques in \cite{CG}, J. Stevens \cite{Stevens2} calculated the dimension of the higher tangent sheaves $\sT_X^i$
for rational surface singularities. 

It seems the methods in \cite{CG}, \cite{Stevens2} work only for rational surface singularities. Therefore it is interesting to generalize the method there so that it will work for 
minimally elliptic singularities as well. 
\end{rmk}

\begin{thm}\label{thm_ellitptic_surface_singularities_intro}(Theorem \ref{thm_ellitptic_surface_singularities})
Let $(X,0)$ be a minimally elliptic  surface singularity of embedding dimension $n+1\ge 5$, and 
$f$ is a general hypersurface section.  Then 
\begin{enumerate}
\item $\dim_{\cc}(\sT_X^2/\m\sT^2_X)=\frac{1}{2}(n+1)(n-4)$, where $\m$ is the maximal ideal and 
$\m\sT_X^2=f\sT_X^2$.
\item $\dim_{\cc}(\sT_X^3/\m\sT^3_X)=\frac{1}{6}(n+1)(n-3)(n-4)$.
\item $\dim_{\cc}(\sT_X^4/\m\sT^4_X)=\frac{1}{12}(n+1)(n-4)(n-3)(n-2)$.
\end{enumerate}
\end{thm}

The calculation of Theorem \ref{thm_rational_surface_intro} and Theorem \ref{thm_ellitptic_surface_singularities_intro}  is based on the results in \cite{BC}, \cite{BC2} and the reduction to $(Y,0)$  using the hyperplane section $f\in \mathfrak{m}_X$.

\subsection{Application to s.l.c. surfaces}\label{subsec_slc_surface}

The calculation of the higher obstruction space for surface singularities is helpful for deforming semi-log-canonical (s.l.c.) surfaces. 
Semi-log-canonical surfaces are surfaces with semi-log-canonical singularities, see \cite{Kollar-Shepherd-Barron}. 
The interesting s.l.c. surface singularities, except the locally complete intersection (l.c.i.) ones, are given by 
cyclic surface quotient singularities and $\zz_2, \zz_3,\zz_4,\zz_6$ quotients of  non l.c.i. simple elliptic singularities, cusp and degenerate cusp singularities, see \cite[Theorem 4.24]{Kollar-Shepherd-Barron}.  
Cyclic surface quotient singularities are rational surface singularities, while simple elliptic singularities and  cusp singularities are minimally elliptic surface singularities. 
 
The moduli space of s.l.c. surfaces was constructed in \cite{Kollar-Book} by adding s.l.c. surfaces on the boundary of Gieseker moduli of general type surfaces.  Let us fix the invariants 
$K^2, \chi$, and let $M:=\overline{M}_{K^2, \chi}$ be the moduli stack of s.l.c. surfaces $X$ with $K_X^2=K^2, \chi(\sO_X)=\chi$.  The moduli stack  $M:=\overline{M}_{K^2, \chi}$ is defined by the moduli functor of 
$\qq$-Gorenstein deformation of s.l.c. surfaces. 
From \cite{Kollar-Book} the moduli stack $M$ is a projective Deligne-Mumford stack. 

For any s.l.c. surface germ $(X,0)$, there is an index one cover $(Z,0)$ which admits a cyclic group $\zz_N$ action, where $N$ is the index of $X$.   All the s.l.c. surface germs glue together and give an index one covering Deligne-Mumford stack 
$\pi: \XX\to X$ such that locally the stack is given by $[Z/\zz_N]$.  Hacking \cite{Hacking} shows that the $\qq$-Gorenstein deformation of the s.l.c. surface $X$ is equivalent to the  deformation of the Deligne-Mumford stack $\XX$.  
Let $\ll^{\bullet}_{\XX}$ be the cotangent complex of $\XX$ in \cite{Illusie1}, \cite{Illusie2}.  Let $A$ be a $\cc$-algebra and  $J$ be a finite $A$-module. 
If  $\sX/A$ is a family of $\qq$-Gorenstein deformation of $X$ over $A$ and $\XX/A$ the family of the corresponding  index one  covering Deligne-Mumford  stacks, then  we define
$T_{\QG}^i(\sX/A,  J):=\Ext^i(\ll^{\bullet}_{\XX/A},  \sO_{\XX}\otimes_{A} J)$.  Note that when $A=\cc, J=\sO_X$, we get the space $T_{\QG}^i(X)=T_{\QG}^i(X,\sO_X)$. 
Then for a s.l.c. surface $X$, the $\qq$-Gorenstein deformation space of $X$ is given by the group $T_{\QG}^1(X)$, and the obstruction space is $T_{\QG}^2(X)$.  

The index one covering Deligne-Mumford stack $\pi: \XX\to X$ is canonical.  In \cite{Jiang_2021}, the author proves that  there is an index one covering Deligne-Mumford stack $f: M^{\ind}:=\overline{M}^{\ind}_{K^2, \chi}\to M$ over the moduli stack 
$M:=\overline{M}_{K^2, \chi}$.  The morphism $f$ is \'etale.  The moduli stacks $M^{\ind}$ and $M$ are both fine moduli stacks.  Let $p^{\ind}: \sM^{\ind}\rightarrow M^{\ind}$ be the universal family, which is a projective, flat and relative Gorenstein morphism between Deligne-Mumford stacks.  This is because the relative  dualizing sheaf $\omega_{\sM^{\ind}/M^{\ind}}$
of $p^{\ind}$  is invertible. 
Let 
$$E^{\bullet}_{M^{\ind}}:=Rp^{\ind}_{*}(\ll^{\bullet}_{\sM^{\ind}/M^{\ind}}\otimes \omega_{\sM^{\ind}/M^{\ind}})[-1]$$
where $\ll^{\bullet}_{\sM^{\ind}/M^{\ind}}$ is the relative cotangent complex of $p^{\ind}$.   This is the case of moduli of projective Deligne-Mumford stacks satisfying the condition in  \cite[Proposition 6.1]{BF}. 
Thus the Kodaira-Spencer map $\ll^{\bullet}_{\sM^{\ind}/M^{\ind}}\to (p^{\ind})^{*}\ll_{M^{\ind}}^{\bullet}[1]$ induces an obstruction theory 
\begin{equation}\label{eqn_OT_intro}
\phi^{\ind}: E^{\bullet}_{M^{\ind}}\to \ll_{M^{\ind}}^{\bullet}
\end{equation}
on $M^{\ind}$.

From Theorem \ref{thm_ellitptic_surface_singularities_intro}, for germ cusp singularities  $(X,p)$, the higher obstruction spaces $\sT^i_{X}$ do not vanish for $i\ge 3$ making the higher obstruction spaces 
$T_{\QG}^i(X)\neq 0$ for $i\ge 3$ by local to global spectral sequence, see \S \ref{subsec_slc}.  Therefore we have

\begin{thm}\label{thm_M_BF_LT_intro}(Theorem \ref{thm_M_BF_LT})
Let $M:=\overline{M}_{K^2, \chi}$ be the main  connected component of the moduli stack of s.l.c. surfaces containing smooth surfaces with $K_S^2=K^2, \chi(\sO_S)=\chi$, and let $M^{\ind}\to M$ be the index one covering Deligne-Mumford stack. 

If there exist s.l.c. surfaces $X$ on the boundary of 
$M$ such that $X$ contains the following surface singularities: 
\begin{enumerate}
\item simple elliptic singularities $(X,0)$ with embedded dimension $\ge 6$, or  the $\zz_2, \zz_3, \zz_4, \zz_6$-quotient of simple elliptic   singularities $(X,0)$ with local embedded dimension $\ge 6$;
\item
cusp germ singularities $(X,0)$ with local embedded dimension $\ge 6$, or the $\zz_2$-quotient cusp  singularities $(X,0)$ with local embedded dimension $\ge 6$.
\end{enumerate} 
then the higher obstruction spaces $T_{\QG}^i(X)\neq 0$ for $i\ge 3$. Thus there is no Li-Tian and Behrend-Fantechi style perfect obstruction theory for the moduli stack $M$, and there is  no such style virtual fundamental class on $M$. 
\end{thm}

Theorem \ref{thm_M_BF_LT_intro} implies that in order to get virtual fundamental class for the moduli stack $M:=\overline{M}_{K^2, \chi}$,  we need new ideas to control the higher obstruction spaces, see \cite{Jiang_2021}.

\subsection{Short outline}
Here is the outline.  In \S \ref{sec_cotangent_hypersurface} we review the construction of cotangent complex and general hypersurface section.  The construction of rational partition curves and elliptic partition curves is given in \S \ref{sec_rational_elliptic_curves}. We calculate some  higher obstruction spaces for rational and minimally elliptic surface singularities in \S \ref{sec_surface_singularities} and talk about the application to the moduli space of s.l.c. surfaces.

%%%-----------------------------------------------------------------------
\subsection*{Conventions}
In this paper we work entirely  over the field of
complex numbers.

%%% ----------------------------------------------------------------------
\subsection*{Acknowledgments}

The author would like to thank Professor P. Hacking for the reference of K. Behnke and J. A. Christophersen on the calculation of obstruction space for surface singularities.  This work is partially supported by 
Simon Collaboration Grant. 

%%% ----------------------------------------------------------------------

\section{Cotangent complex and hypersurface sections}\label{sec_cotangent_hypersurface}

\subsection{Cotangent complex}\label{subsec_cotangent_complex}

Let us recall the construction of cotangent complex for schemes or Deligne-Mumford stacks.  For detail construction, see
\cite{Illusie1}, \cite{Illusie2}, \cite{Stack_Project}, and \cite{LS}. 

All rings are Noetherian.  Let $A$ be a commutative ring, and $B$ is an $A$-algebra. The cotangent complex  $\ll^{\bullet}_{B/A}$ is constructed as follows:
Let $R$ be a regular $A$-algebra such that $B=R/I$ for some ideal $I\subset R$.  
Let 
$$\cdots\rightarrow E_n\stackrel{e_n}{\rightarrow} \cdots\stackrel{e_4}{\rightarrow} E_3\stackrel{e_3}{\rightarrow} E_2\stackrel{e_2}{\rightarrow} E_1\stackrel{e_1}{\rightarrow} R\stackrel{e_0}{\rightarrow} B\to 0$$
be a resolution of $B=R/I$ by free $R$-modules.   Then the cotangent complex  $\ll^{\bullet}_{B/A}$ is given by
\begin{equation}\label{eqn_cotangent_complex}
\cdots\rightarrow \Omega_{E_n}\otimes_{R}B\stackrel{d_n}{\rightarrow} \cdots\stackrel{d_4}{\rightarrow} \Omega_{E_3}\otimes_{R}B\stackrel{d_3}{\rightarrow} \Omega_{E_2}\otimes_{R}B\stackrel{d_2}{\rightarrow} \Omega_{E_1}\otimes_{R}B\stackrel{d_1}{\rightarrow} \Omega_{R}\otimes_{R}B\to 0
\end{equation}

\begin{rmk}\label{rmk_cotangent_complex_LS}
If the resolution is given by: 
$$0\to E_2\stackrel{e_2}{\rightarrow} E_1\stackrel{e_1}{\rightarrow} R\stackrel{e_0}{\rightarrow} B\to 0$$
then $E_2$ is a naturally a $B$-module since $I\cdot E_2=0$ (for any $a\in I$, $x\in E_2$, since 
$\Im(e_1)=\ker(e_0)=I$, there exists a $b\in E_1$, $e_1(b)=a$, and $ax=e_1(b)x$. Therefore $e_2(ax)=e_1(b)e_2(x)=e_1e_2(x)b=0$ implies 
$ax=0$.)
The cotangent complex is actually given by: 
$$\ll^{\bullet}_{B/A}:  0\to E_2\stackrel{d_2}{\rightarrow} E_1\otimes_{R}B\stackrel{d_1}{\rightarrow} \Omega_{R}\otimes_{R}B\to 0,$$
where $d_2$ is induced from $e_2$,  $d_1=d\circ (e_1\otimes_{R}B)$ and 
$d: I/I^2\to \Omega_{R}\otimes_{R}B$ is the natural differential, 
see \cite[Section 2.1]{LS}. 
\end{rmk}

The tangent cohomology module is defined as:
$$T^i(B/A, M):=H^i(\Hom_{B}(\ll^{\bullet}_{B/A}, M)).$$
If $X$ is surface germ singularity $(X,0)$, then $\sO_X$ is the structure sheaf and we will denote 
$$\sT_{X}^i:=T^i(\sO_X/\cc, \sO_X)$$ 
which is the tangent complex of the local singularity $(X,0)$.

\subsection{Properties}

The modules $T^i(B/A, M)$ have the following properties:

\subsubsection{Base change}\label{subsubsec_base_change}

Consider the following commutative diagram:
\[
\xymatrix{
A\ar[r]\ar[d]_{\phi}& A^\prime\ar[d]\\
B\ar[r]& B^\prime
}
\]
where $\phi$ is flat and $M^\prime$ is a $B^\prime$-module.  Then we have a natural isomorphism
$$T^i(B^\prime/A^\prime, M^\prime)\cong T^i(B/A, M).$$
Moreover, if $A^\prime$ is a flat $A$-module, then 
$$T^i(B^\prime/A^\prime, M\otimes_{A}A^\prime)\cong T^i(B/A, M)\otimes_{A}A^\prime.$$

\subsubsection{Zariski-Jacobi exact sequence}\label{subsubsec_Zariski-Jacobi}

Let $A\to B\to C$ be ring homomorphisms, and $M$ is a $C$-module, then there is a long exact sequence:
$$
\begin{array}{ll}
\cdots & \rightarrow T^i(C/B, M)\rightarrow T^i(C/A, M)\rightarrow T^i(B/A, M)\\
& \rightarrow T^{i+1}(C/B, M)\rightarrow \cdots
\end{array}
$$

\subsubsection{Long exact sequence}\label{subsubsec_long_sequence}

If $0\to M^\prime\rightarrow M\rightarrow M^{\prime\prime}\to 0$ is a short exact sequence of 
$B$-modules, then there is a long exact sequence:
$$
\begin{array}{ll}
\cdots & \rightarrow T^i(B/A, M^\prime)\rightarrow T^i(B/A, M)\rightarrow T^i(B/A, M^{\prime\prime})\\
& \rightarrow T^{i+1}(B/A, M^\prime)\rightarrow \cdots
\end{array}
$$
Other properties of the cotangent complex can be found in \cite{Illusie1}, \cite{Illusie2}.  In particular, if $B$ is a smooth $A$-algebra, then 
$T^i(B/A, M)=0$ for $i\ge 1$ and for all $B$-modules $M$.

\subsection{Hypersurface sections}\label{subsec_hypersurface_section}

Let $(X,0)$ be a germ of analytic space of positive dimension.  Let $\sO_X=\sO_{X,0}$ be the local ring of $X$ at $0$, and 
$f\in \sO_X$.  From \cite{BC}, $Y=f^{-1}(0)$ is a hypersurface section on $X$ if the element 
$f\in \m_{X}$, the maximal ideal of the germ at $0$,  and $f$ is a non-zero divisor in $\sO_X$. 
Let us suppose $(Y,0)$ is an isolated singularity. 

Let $f: X\to (\cc,0)$ be a flat deformation of $Y$, i.e., there exists a diagram:
\[
\xymatrix{
X\ar[r]\ar[d]_{f}& \YY\ar[d]^{\pi}\\
(\cc,0)\ar[r]^{j}& S
}
\]
such that $\pi$ is a versal deformation of $Y$.  We let $\sO_1=\cc\{t\}$ be the local ring of $(\cc, 0)$. Then 
$\sO_X$ is a $\sO_1$-module via $f^*: \sO_1\to \sO_{X,0}$
mapping $t$ to $f$. There is a short exact sequence
$$0\to \sO_X\stackrel{f}{\rightarrow} \sO_X\rightarrow \sO_Y\to 0.$$
From \S \ref{subsubsec_long_sequence}, we have a long exact sequence:
\begin{equation}\label{eqn_long_exact_sequence}
\begin{array}{ll}
\cdots &\rightarrow T^1(\sO_X/\sO_1, \sO_X)\stackrel{f}{\rightarrow} T^1(\sO_X/\sO_1, \sO_X)\stackrel{\alpha_1}{\rightarrow} T^1(\sO_X/\sO_1, \sO_Y)\\
& \rightarrow T^2(\sO_X/\sO_1, \sO_X)\stackrel{f}{\rightarrow} T^2(\sO_X/\sO_1, \sO_X)\stackrel{\alpha_2}{\rightarrow} T^2(\sO_X/\sO_1, \sO_Y)\\
& \rightarrow T^3(\sO_X/\sO_1, \sO_X)\stackrel{f}{\rightarrow} T^3(\sO_X/\sO_1, \sO_X)\stackrel{\alpha_3}{\rightarrow} T^3(\sO_X/\sO_1, \sO_Y)\\
& \rightarrow \cdots \\
& \rightarrow T^i(\sO_X/\sO_1, \sO_X)\stackrel{f}{\rightarrow} T^i(\sO_X/\sO_1, \sO_X)\stackrel{\alpha_i}{\rightarrow} T^i(\sO_X/\sO_1, \sO_Y)\\
& \rightarrow T^{i+1}(\sO_X/\sO_1, \sO_X)\rightarrow \cdots 
\end{array}
\end{equation}
By base change in \S \ref{subsubsec_base_change}, $T^i(\sO_X/\sO_1, \sO_Y)\cong \sT^i_{Y}$.
By the exact sequence 
$$\cc\to \sO_1\to \sO_X$$
from \S \ref{subsubsec_Zariski-Jacobi}, we have an exact sequence
$$
\begin{array}{ll}
\cdots & \rightarrow T^i(\sO_X/\sO_1, \sO_X)\rightarrow \sT_X^i\rightarrow T^i(\sO_1/\cc,\sO_X)\\
& \rightarrow  T^{i+1}(\sO_X/\sO_1, \sO_X)\rightarrow \cdots
\end{array}
$$
The ring $\sO_1$ is regular,  therefore $T^i(\sO_1/\cc,\sO_X)=0$ for $i\ge 1$ and thus 
$T^i(\sO_X/\sO_1,\sO_X)\cong T^i(\sO_X/\cc,\sO_X)=\sT^i_{X}$ for $i\ge 2$. 

We assume that $(X,0)$ is an isolated singularity. Then $f: X\to (\cc,0)$ is a smoothing.  In this case a generic point $j(t)$ is smooth and let $e_f$ be the dimension 
$\dim_{\cc}(\Im (\alpha_1))$ for $\alpha_1: T^1(\sO_X/\sO_1, \sO_X)\to T^1(\sO_X/\sO_1, \sO_Y)$, 
which is the dimension of the Zariski tangent space of $S$ at the generic point of the curve $j(\cc)$, see \cite[Corollary 2.2.2]{Greuel-Looijenga}. 
In the isolated singularity case, $e_f$ is the dimension of the irreducible smoothing component containing 
$j(\cc,0)$. 

From above analysis and \cite[\S 1.3.2 Main Lemma]{BC}, 

\begin{thm}\label{thm_main_thm}
Let $(X,0)$ be an isolated singularity, and $Y=f^{-1}(0)$ is a hypersurface section of $X$. Then
\begin{enumerate}
\item there are exact sequences:
$$T^1(\sO_X/\sO_1, \sO_X)\stackrel{\alpha_1}{\rightarrow} \sT^1_{Y}\rightarrow \sT_X^2\stackrel{f}{\rightarrow} \sT_X^2\rightarrow \Im(\alpha_2)\to 0$$
and 
$$0\to \Im(\alpha_{i-1}){\rightarrow} \sT^{i-1}_{Y}\rightarrow \sT_X^{i}\stackrel{f}{\rightarrow} \sT_X^{i}\rightarrow \Im(\alpha_i)\to 0$$
for $i\ge 3$.
\item  $\dim_{\cc}(\sT_X^2/f\sT_X^2)-\rk_{\sO_1}(\sT_X^2)=\dim_{\cc}(\sT_Y^1)-e_f$. 
\item If $f$ is a smoothing of $Y$, then $e_f$ is the dimension of the smoothing component of $f$ and 
$$\dim_{\cc}(\sT_X^2/f\sT_X^2)=\dim_{\cc}(\sT_Y^1)-e_f.$$
\item Under the same condition as in (3),  
$$\dim_{\cc}(\sT_X^i/f\sT_X^i)=\dim_{\cc}(\sT_Y^{i-1})-\dim_{\cc}(\sT_X^{i-1}/f\sT_X^{i-1})$$
for $i\ge 3$. 
\end{enumerate}
\end{thm}
\begin{proof}
(1), (2) and (3) are results in \cite[\S 1.3.2 Main Lemma]{BC}.  (4) is a direct consequence of the long exact sequence (\ref{eqn_long_exact_sequence}). 
\end{proof}

\begin{rmk}
From \cite[\S 5]{Buchweitz}, it is known that an isolated singularity $(X,0)$ is unobstructed if the versal base space 
 is smooth.  The tangent space $\sT_X^2=0$ if and only if $(X,0)$ is unobstructed and every hypersurface section is unobstructed. 
 A smoothable singularity is unobstructed if and only if it is a hypersurface section of an isolated singularity $X$ with $\sT^2_X=0$.
\end{rmk}

\section{Rational partition curves and elliptic partition curves}\label{sec_rational_elliptic_curves}

In this section we recall two types of partition curve singularities: rational partition curves and elliptic partition curves.  They are exactly the hypersurface sections of rational surface 
singularities and minimally elliptic surface singularities. 

\subsection{Rational partition curves}\label{subsec_rational_partition}

\subsubsection{Notations}

We first classify some notations.   The Cohen-Macaulay type $t$ of a one-dimensional Cohen-Macaulay singularity $Y\subset (\cc^n, 0)$ is the minimal number of generators of the dualizing module 
$\omega_Y$.  Assume that $(Y,0)$ is a reduced curve singularity, where $\sO_Y$ is Cohen-Macaulay. Let 
$$\nu: \widetilde{Y}\to Y$$
be the normalization map. 
The length of the Artinian module $\nu_*(\sO_{\widetilde{Y}})_0/\sO_Y$ is the $\delta$-invariant of $Y$.  The Milnor number $\mu$ is the length of 
$\omega_Y/d\sO_Y$, where $d: \sO_Y\to \Omega_Y\to \omega_Y$ is the composition map. From \cite[Proposition 1.2.1]{Buchweitz-Greuel}, if 
$Y$ has $r$ branches, $\mu=2\delta-r+1$. 

\subsubsection{The wedge}

Let us recall the wedge $Y_1\vee Y_2$ for two curve singularities.  Let $I_1, I_2$ be the ideals of $Y_1, Y_2$ in $\cc^n$ respectively.  The intersection number is defined as
$$i(Y_1, Y_2)=\dim_{\cc}(\sO_{\cc^n,0}/(I_1+I_2))$$
if $Y_1\vee Y_2$ has no common component. 
$Y$ is decomposable or $Y=Y_1\vee Y_2$ if the Zariski tangent spaces of $Y_1$ and $Y_2$ have only the point $0$ in common.  The wedges have the following properties:

\begin{lem}\label{lem_wedges}(\cite[Lemma 3.1.3]{BC})
\begin{enumerate}
\item For $Y, Y_1, Y_2$,  $Y=Y_1\vee Y_2$ if and only if $i(Y_1, Y_2)=1$. This is from \cite{Hironaka}. 
\item Let $\delta, \delta_1, \delta_2$ be the $\delta$-invariants of $Y, Y_1, Y_2$. Then 
$\delta=\delta_1+\delta_2+i(Y_1, Y_2)$.
\item  Let $Y=Y_1\cup\cdots\cup Y_r$ be a union of $r$ curves with $\delta$-invariants $\delta_1, \cdots, \delta_r$. Then 
$$\delta\ge \delta_1+\cdots +\delta_r+r-1$$
and the equality holds if and only if $Y=Y_1\vee\cdots\vee Y_r$. 
\item Let $Y=Y_1\vee Y_2$ and assume that $Y_1$ and $Y_2$ have Cohen-Macaulay type $t_i$ (if $Y_i$ is a smooth branch of $Y$ we let $t_i=0$). Then the Cohen-Macaulay 
type $t$ of $Y$ is $t=t_1+t_2+1$.
\end{enumerate}
\end{lem}

\subsubsection{Rational partition curve}\label{subsubsec_rational_partition}

First let us recall a result in \cite[Theorem 2.5 (3)]{Greuel}. Let $(Y,0)$ be a curve singularity. If it is quasi-homogeneous and smoothable, then the dimension of the smoothing 
components in the base space of the semi-universal deformation is $e=\mu+t-1$. 

The partition curve is defined as follows. 

\begin{defn}\label{defn_rational_partition_1}
Let $m\in \nn$ be an positive integer, define $Y(m)\subset \cc^m$ to be the monomial curve singularity with semi-group generated by 
$\{m, m+1, \cdots, 2m-1\}$.  In other words, $Y(m)$ is given by $\nu: \cc\to \cc^n$ sending $t$ to 
$\nu(t)=(t^m, \cdots, t^{2m-1})$. The ideal of $Y(m)\subset \cc^m$ in $\cc[z_1, \cdots, z_m]$ is minimally generated by all the $2\times 2$-minors of the $2\times m$-matrix:
$$
\mat{ccccc}
z_1&z_2&\cdots&z_{m-1}& z_m\\
z_2&z_3&\cdots&z_{m}& z_1^2
\rix
$$
\end{defn}

\begin{defn}\label{defn_rational_partition_2}
Let $m=m_1+\cdots+m_r$ be a partition of $m$.  The rational partition curve associated with $(m_1,\cdots, m_r)$ is the wedge  
$$Y(m_1,\cdots, m_r)=Y(m_1)\vee Y(m_2)\vee\cdots\vee Y(m_r),$$
where $Y(m_i)\subset \cc^{m_i}\subset \cc^m$. 
\end{defn}

\begin{example}
Here are some basic examples.

$\bullet$ $Y(1)=\aaa_{\cc}^1$ is a line.

$\bullet$ $Y(2)\subset \cc^2$ is given by the ideal $I=(y^3-x^2)$.

$\bullet$ $Y(1,1)$ is a nodal curve. 

$\bullet$ $Y(3)\subset \cc^3$ is given by $I=(z_1z_3-z_2^2, z_1^2z_2-z_3^2, z_1^3-z_2z_3)$.

$\bullet$ $Y(2,1)=Y(2)\vee Y(1)$. 

$\bullet$ $Y(1,1,1)\subset \cc^3$ is given by $I=(xy, yz, xz)$.
\end{example}

The equations of the rational partition curve $Y(m_1, \cdots, m_r)\subset \cc^m$ is given as follows.  Let 
$$z_1^{(1)}, \cdots z_{m_1}^{(1)}, z_1^{(2)}, \cdots, z_{m_2}^{(2)}, \cdots,  z_1^{(r)}, \cdots, z_{m_r}^{(r)}$$
be the coordinates of $\cc^m$ and let 
$$
M_i=
\mat{ccccc}
z_1^{(i)}&z_2^{(i)}&\cdots&z_{m_{i-1}}^{(i)}& z_m^{(i)}\\
z_2^{(i)}&z_3^{(i)}&\cdots&z_{m_i}^{(i)}& (z_1^{(i)})^2
\rix
\quad  i=1, \cdots, r, \quad  m_i>1.
$$
Then $Y(m_1, \cdots, m_r)\subset \cc^m$ is given by the ideal generated by
$$
\begin{cases}
\rk(M_i)\leq 1;  &  i=1, \cdots, r, m_i>1;\\
z_i^{(k)}z_{j}^{(l)}=0;  &    k\neq l, 1\leq i\leq m_k,  1\leq j\leq m_l. 
\end{cases}
$$

\subsubsection{Main result}

The main result we review in this section is:

\begin{prop}\label{prop_rational_partition}(\cite[Proposition  3.2.3]{BC})
For a partition $m=m_1+\cdots+m_r$, the rational partition curve $Y(m_1, \cdots, m_r)\subset \cc^m$ has the following invariants:
\begin{enumerate}
\item $\delta=m-1$.
\item The Cohen-Macaulay type is $t=m-1$.
\item The Milnor number is $\mu=2m-r-1$.
\item The smoothing components have dimension $3(m-1)-r$. 
\end{enumerate}
\end{prop}

\begin{prop}\label{prop_rational_partition_2}(\cite[\S 3.3]{BC})
Let $Y$ be a reduced curve singularity with embedding dimension $n$. Then 
\begin{enumerate}
\item The $\delta$-invariant of $Y$ is at least $n-1$.
\item $\delta=n-1$ if and only if $Y$ is a rational partition curve. 
\end{enumerate}
\end{prop}

\subsection{Elliptic partition curves}\label{subsec_elliptic_partition}

A in \cite[\S 3.1]{BC2}, a Gorenstein isolated curve singularity $(Y,0)$ of embedding dimension $n (n\ge 3)$ has $\delta$-invariant $\delta\ge n+1$. 
The Gorenstein curve singularity  $(Y,0)$  with $\delta=n+1$ is called the elliptic partition curve. 

Let us construct the elliptic partition curve.  Fix $n\ge 2$ an integer.  The elliptic partition curve is given by 
$Y_{\vec{p}}$ for $\vec{p}=(p_1, \cdots, p_r)$ a partition of $n+1$, and has embedding dimension $n$ and 
$\delta$-invariant $\delta=n+1$.  If $\vec{p}=(n+1)$, then $Y_{\vec{p}}$ is the monomial curve with semi-group 
$$\{0,n+1,\cdots, 2n,2n+2,\cdots\}$$
with gap set $\{1, 2,\cdots, n, 2n+1\}$.  It is isomorphic to 
$\spec(\cc[t^{n+1},\cdots, t^{2n}, t^{2n+2}])=\spec(\cc[z_1,\cdots, z_n]/I)$ where $I$ is the monomial ideal. 
For instance, when $n=2$, $\vec{p}=(2+1)$, $Y_{\vec{p}}\subset \cc^2$ is $\spec(\cc[x, y]/(y^3-x^4))$. 

If $\vec{p}=(p_1, \cdots, p_r)$  with $r\ge 2$, we construct $Y_{\vec{p}}$ as follows.  We first have $Y(p_i)$ for each $p_i$ as before in \S \ref{subsubsec_rational_partition}, where 
$$Y_i=Y(p_i)\subset \cc^{p_i}\cong V_i\subset \cc^n$$
for a linear subspace $V_i\subset \cc^n$ in general position. Then 
$$Y_{\vec{p}}=Y_1\cup\cdots\cup Y_r.$$
More precisely, $Y_{\vec{p}}$ can be constructed as follows.   First let 
$$Y^\prime=Y_1\vee\cdots\vee Y_{r-1}\subset \cc^{p_1}\oplus\cdots\oplus \cc^{p_{r-1}}$$
and we put $Y_r\subset\cc^{p_r}$ in general position relative to $\cc^{p_1},\cdots, \cc^{p_{r-1}}$. 
General position means that the projection of the tangent line of $Y_r$ to $\cc^{p_i}$ does not map into the osculating hyperplane of $Y_i$ which is the hyperplane 
generated by the first $p_i-1$ basis vectors in the standard representation of $Y_i$ for $i=1, 2, \cdots, r-1$. 
$Y_{\vec{p}}$ has embedding dimension $n$.  The $\delta$-invariant of $Y^\prime$ is 
$$\delta^\prime=\sum_{i=1}^{r-1}(p_i-1)+r-2$$
and $\delta(Y_r)=p_r-1$, 
$i(Y^\prime, Y_r)=2$. Thus this implies $\delta(Y_{\vec{p}})=n+1$. 

\begin{prop}\label{prop_elliptic_partition}(\cite[Proposition 3.4.1]{BC2})
Any isolated Gorenstein curve singularity of embedding dimension $n$ and minimal $\delta$-invariant $n+1$ is isomorphic to an elliptic partition curve
$Y_{\vec{p}}$ for $\vec{p}=(p_1, \cdots, p_r)$ a partition of $n+1$. 
\end{prop}

\section{Obstruction space for surface singularities}\label{sec_surface_singularities}

\subsection{Hypersurface sections and surface singularities}\label{subsec_surface_singualrity_hypersurface}

Let $(X,0)$ be a surface singularity, and $f\in \m\subset \sO_{X,0}$ an element in the local ring $\sO_{X,0}$ such that $f$ projects to a sufficient general 
element of $\m/\m^2$.  Then $Y=f^{-1}(0)$ is a general hypersurface section of $X$. The embedding dimension of $Y$ is one less than the embedding dimension 
of $X$. 

Let $\pi: \widetilde{X}\to X$ be a resolution of surface singularity $X$.  We assume that the resolution $\pi$ factors through the blow-up of the maximal ideal 
$\m$ of $X$ at $0$., (i.e., $\pi^{-1}(\m)\cdot \sO_{\widetilde{X}}=\sO_{\widetilde{X}}(-Z)$ is locally free, where $Z$ is the fundamental cycle), and that the reduced 
exceptional divisor is a union of smooth projective curves. If 
$Y$ is a general hypersurface section, we have 
$$\pi^{-1}(Y)=Z+\widetilde{Y},$$
where $\widetilde{Y}=\overline{\pi^{-1}(Y\setminus \{0\})}$ is a disjoint union of smooth curves meeting the exceptional divisor transversely in smooth points. 

Let us recall the main result in \cite[Theorem 4.1.2]{BC}.

\begin{thm}\label{thm_hypersurface_rational}
A general hypersurface section $(Y,0)$ of a rational surface singularity $(X,0)$ of embedding dimension $n+1\ge 3$ is a rational partition curve for some partition 
$n=n_1+\cdots+n_r$.
\end{thm}
\begin{proof}
The proof uses the resolution  $\pi: \widetilde{X}\to X$, see \cite[\S 4.2.1]{BC} and we omit the details. 
\end{proof}

Here is a corollary from above theorem. 

\begin{cor}(\cite[Corollary 4.2.2]{BC})
\begin{enumerate}
\item If $(X,0)$ is a rational surface singularity, and $f\in\m$ is general, then $\delta=-Z^2-1=n-1$, where 
$Z\subset \widetilde{X}$ is the fundamental cycle. 
\item If $(X,0)$ is a minimally elliptic singularity, and $-Z^2\ge 2$ and $f$ is general, then $\delta=-Z^2=n$. 
\end{enumerate}
\end{cor}

This implies that for minimally elliptic singularities, we have

\begin{prop}\label{prop_hypersurface_elliptic}(\cite[Proposition 4.1.1]{BC2})
Let $(X,0)$  be a minimally elliptic surface singularity of multiplicity $n\ge 3$. The general hypersurface section $(Y,0)$ of $(X,0)$ is an elliptic partition curve with 
$n=n_1+\cdots+n_r$. 
\end{prop}

\subsection{Tangent dimension of rational partition curves}\label{subsec_tangent_dimension_rational}

Let $(Y,0)$ be a rational partition curve singularity of type $n=(n_1,\cdots, n_r)$.  Let 
$f: (X,0)\to \cc$ be a smoothing of $Y$. Let $\tau=\dim_{\cc}\sT_Y^1$,  and $e$ be the dimension of the smoothing components.  Then from \cite[Proposition 4.5]{BC}, 
$$\tau-e=\max(0, n(n-3))$$
which is only dependent to the embedding dimension $n$ and independent to the partition curve.  

\begin{prop}\label{prop_dimension_TY}
We have 
\begin{enumerate}
\item $\dim_{\cc}\sT_Y^1=n(n-1)-r$.
\item  $\dim_{\cc}(\sT_Y^2)=\frac{1}{2}n(n-1)(n-3)$.
\item  $\dim_{\cc}(\sT_Y^3)=\frac{1}{6}n(n-1)(n-2)(2n-5)$.
\end{enumerate}
\end{prop}
\begin{proof}
The proof is based on the hypersurface section of the partition curve $(Y,0)$ and Theorem \ref{thm_main_thm}. Let us recall it here. 

We fix an integer $r\ge 3$, and let 
$$B_r=\cc\{x_1,\cdots, x_r\}/\m^2$$
be the Artinian local $\cc$-algebra of embedding dimension $r$ and minimal multiplicity $r+1$.  Let $Z_r$ be the corresponding zero subscheme of the Artinian local 
algebra. Then $Z_r$ is a general hypersurface section of a partition curve of embedding dimension $r+1$. 

As proved in \cite[\S 6.2.1]{BC}, choose a smoothing 
$$f: Y\to (\cc,0)$$
of the point scheme $Z_r$ there is an exact sequence
$$0\to \Der_{\sO_1}(\sO_Y, \sO_Y)\stackrel{f}{\rightarrow}\Der_{\sO_1}(\sO_Y, \sO_Y)\stackrel{\beta}{\rightarrow}\Der_{\cc}(\sO_{Z_r}, \sO_{Z_r})$$
By Wahl's conjecture in \cite{Greuel-Looijenga}, the length $\cok(\beta)$ is the dimension of the smoothing component. 
Also sine $\Der_{\sO_1}(\sO_Y, \sO_Y)=0$,  this is just $\Der_{\cc}(\sO_{Z_r}, \sO_{Z_r})$, which is the derivations of the Artinian algebra $B_r$. 
This space is the endomorphism of the maximal ideal $\m$ as a vector space which is $r^2$.  Using the cotangent complex of $B_r$ as in Remark \ref{rmk_cotangent_complex_LS}, 
in \cite[Proposition 6.2.1]{BC}, the authors calculated 
\begin{enumerate}
\item  $\dim_{\cc}\sT_{Z_r}^0=r^2$, i.e., the dimension of the smoothing component of $Z_r$.
\item  $\dim_{\cc}\sT_{Z_r}^1=\frac{1}{2}(r-1)r(r+2)$.
\item $\dim_{\cc}\sT_{Z_r}^2=\frac{1}{6}r(r+1)(2r^2-2r-3)$.
\end{enumerate}

Now note that $r+1=n$,  we apply Theorem \ref{thm_main_thm}.  The first formula is from \cite[Proposition 4.5.2]{BC}.  
For the second and the third:
\begin{align*}
\dim_{\cc}\sT_Y^2&=\dim_{\cc}(\sT^1_{Z_r})-\dim_{\cc}(\sT_Y^1)\\
&=\frac{1}{2}(r-1)r(r+2)-r^2\\
&=\frac{1}{2}r(r-2)(r+1)\\
&=\frac{1}{2}n(n-1)(n-3).
\end{align*}
and 
\begin{align*}
\dim_{\cc}\sT_Y^3&=\dim_{\cc}(\sT^2_{Z_r})-\dim_{\cc}(\sT_Y^2)\\
&=\frac{1}{6}r(r+1)(2r^2-2r-3)-\frac{1}{2}n(n-1)(n-3)\\
&=\frac{1}{6}n(n-1)(n-2)(2n-5).
\end{align*}
\end{proof}

\begin{rmk}\label{rmk_Br_dimension}
It is interesting to get the dimension $\dim_{\cc}(\sT_{Z_r}^{i})$ for all $i\ge 2$.   In \cite[\S 7]{Avramov},  the dimension of the higher tangent space $\sT_{Z^r}^i$ for $i\ge 3$ is determined by the Betti numbers of the infinite free resolutions of 
$B_r$ and they do not vanish in general.  
\end{rmk}

\subsection{Tangent dimension of elliptic partition curves}\label{subsec_tangent_dimension_elliptic}

Let $Y_{\vec{p}}$ be an elliptic partition curve singularity, where $\vec{p}=(p_1, \cdots, p_r)$ is a partition of $n+1$, 
$\delta=n+1$ and the embedding dimension is $n$.  The general hypersurface section of 
$Y_{\vec{p}}$ is a Gorenstein fat point of multiplicity $r+2=n+1$ and embedding dimension $r=n-1$. 

The algebra structure $A=A_r$  of the Gorenstein fat point is given by a basis $e_0, e_1,\cdots, e_r, e_{r+1}$ where $e_0$ is the unit element, and 
the only  nontrivial products are
$$e_1^2=\cdots=e_r^2=e_{r+1}.$$
From \cite[Proposition 5.2.1]{BC2}, such an Artinian Gorenstein algebra has a pure minimal free resolution.  The generators of the $i$-th module of the syzygies are in 
degree $i+1$ for $1\leq i\leq r-1$, and in degree $r+2$ for $i=r$.  
The $i$-th module of the syzygies is minimally generated by 
$$b_i=\frac{i(r-i)}{r+1}\mat{c}r+2\\i+1\rix$$
elements for $i\leq r-1$, and $b_r=1$.
Also the syzygies tell us that there are $\frac{1}{2}(r-1)(r-2)$ quadratic equations for the singularity with $\frac{1}{3}r(r+2)(r-2)$ linear relations $(r\ge 3)$ and 
$\frac{1}{8}r(r-1)(r-3)(r+2)$ linear second syzygies $(r\ge 4)$. 

\begin{prop}\label{prop_tangent_Gorenstein_fat}(\cite[Proposition 5.3.1]{BC2})
If $A$ has embedding dimension $r\ge 3$, then 
\begin{enumerate}
\item $\dim_{\cc}(T_A^0)=\frac{1}{2}(r^2+r+2)$. 
\item  $\dim_{\cc}(T_A^1)= \frac{1}{6}r(r-1)(r+4)$, and $T_A^1$ is concentrated in degree $-1$. 
\item $\dim_{\cc}(T_A^2)=\frac{1}{12}r(r+1)(r+2)(r-3)$, and  $T_A^2$ is concentrated in degree $-2$. 
\end{enumerate}
\end{prop}
\begin{proof}
The proof used cotangent complex in Remark \ref{rmk_cotangent_complex_LS} to do calculations. 
\end{proof}

\begin{prop}\label{prop_tangent_elliptic_partition}
Let $(Y,0)$ be an elliptic partition curve of multiplicity $n+1$, and with $r$ branches, then 
\begin{enumerate}
\item $\dim_{\cc}(\sT_Y^1)=\frac{1}{2}n(n+1)-r+1$. 
\item  $\dim_{\cc}(\sT_Y^2)= \frac{1}{6}n(n+1)(n-4)$, and $\sT_Y^2$ is  annihilated by the maximal ideal of the local ring.  
\item $\dim_{\cc}(\sT_Y^3)=\frac{1}{12}n(n+1)(n-3)(n-4)$. 
\end{enumerate}
\end{prop}
\begin{proof}
Results (1) and (2) are from 
\cite[Proposition 5.4.1]{BC2}.  For (3), we use Theorem \ref{thm_main_thm} to calculate:
\begin{align*}
\dim_{\cc}\sT_Y^3&=\dim_{\cc}(\sT^2_{A})-\dim_{\cc}(\sT_Y^2)\\
&=\frac{1}{12}r(r+1)(r+2)(r-3)-\frac{1}{6}n(n+1)(n-4)\\
&=\frac{1}{12}n(n+1)(n-3)(n-4).
\end{align*}
\end{proof}

\subsection{Rational and elliptic surface singularities}\label{subsec_rational_elliptic_surface}

In this section we prove the main results for rational surface singularities and minimally elliptic surface singularities. 

Our main results are:

\begin{thm}\label{thm_rational_surface}
Let $(X,0)$ be a rational surface singularity of embedding dimension $n+1\ge 4$, and 
$f$ is a general hypersurface section. 
Let $\pi: \widetilde{X}\to X$ be the minimal resolution.  Let 
$E=\cup E_i$ be the exceptional divisor and $Z$ the fundamental cycle. Then 
\begin{enumerate}
\item $\dim_{\cc}(\sT_X^2/\m \sT_X^2)=(n-1)(n-3)$, where $\m\subset \sT_X^2$ is the maximal ideal of 
$\sT_X^2$ taken as a module over the local ring $\sO_{X,0}$. 
\item If the fundamental cycle $Z$ is reduced and if for any connected subgraph $E^\prime\subset E$ with $Z\cdot E^\prime=0$ the self-intersection number
$E^\prime\cdot E^\prime=-2$ or $-3$, then $\dim_{\cc}(\sT_X^2)=(n-1)(n-3)$.
\item If $X$ is a quotient singularity then $\dim_{\cc}(\sT_X^2)=(n-1)(n-3)$.
\item If $X$ is a quotient singularity then $\dim_{\cc}(\sT_X^3)=\frac{1}{2}(n-1)(n-2)(n-3)$.
\item If $X$ is a quotient singularity then $\dim_{\cc}(\sT_X^4)=\frac{1}{6}(n-1)(n-2)(2n^2-8n+9)$.
\end{enumerate}
\end{thm}
\begin{proof}
The formula in (1), (2), (3) are calculated in \cite[Theorem]{BC} using Theorem \ref{thm_main_thm} and the fact that 
$\m\cdot \sT2_X=0$ in (2), (3).  
Assume that the modules $\sT_X^3, \sT_X^4$ are also annihilated by the maximal ideal $\m$, then using the short exact sequences in \ref{thm_main_thm}
we have
\begin{align*}
\dim_{\cc}\sT_X^3&=\dim_{\cc}(\sT^2_{Y})-\dim_{\cc}(\sT_X^2)\\
&=\frac{1}{2}n(n-1)(n-3)-(n-1)(n-3)\\
&=\frac{1}{2}(n-1)(n-2)(n-3).
\end{align*}

\begin{align*}
\dim_{\cc}\sT_X^4&=\dim_{\cc}(\sT^3_{Y})-\dim_{\cc}(\sT_X^3)\\
&=\frac{1}{6}n(n-1)(n-2)(2n-5)-\frac{1}{2}(n-1)(n-2)(n-3)\\
&=\frac{1}{6}(n-1)(n-2)(2n^2-8n+9).
\end{align*}
\end{proof}

\begin{rmk}
From Remark \ref{rmk_Br_dimension}, the non-vanishing of higher $\sT_{Z_r}^i$ for $i\ge 3$ implies that the higher tangent spaces
$\sT_{X}^i$ also do not vanish. 
\end{rmk}

We provide here a proof that $\dim_{\cc}(\sT_X^3/\m \sT_X^3)=\dim_{\cc}(\sT_X^3)$.  We generalize the argument in \cite[Proposition 2.1.1]{BC}.

\begin{lem}\label{lem_generalized_BC}
Let $f_1,\cdots, f_n, g_1,\cdots, g_n$ be elements of the maximal ideal of the local ring $\sO_{\cc^e,0}$ of $\cc^e$ at $0$. Let $X$ be the space defined by
$$
\rk\mat{cccc}f_1&f_2&\cdots& f_n\\
g_1&g_2&\cdots&g_n\rix
\leq 1
$$
Assume that $X$ is Cohen-Macaulay of codimension $n-1$ in $\cc^e$. Then the $\sO_X$-modules 
$\sT_X^i$ is annihilated by $(f_1,\cdots, f_n, g_1, \cdots, g_n)\sO_X=I\sO_X$. 
\end{lem}
\begin{proof}
The case of $\sT_X^2$ is proved in \cite[Proposition 2.1.1]{BC}. Let us recall it here. 
First let 
$A=\cc[x_1, \cdots, x_e]$ and $F, G$ be free $A$-algebras of rank $n$ and $2$ respectively.  The Eagon-Northcott complex of $\sO_X$ (see, \cite{Eisenbud}) which is a resolution of $A/I$,  is given as follows. 
Let $\alpha: F\to G$ be the homomorphism given by the matrix
$$
\mat{cccc}f_1&f_2&\cdots& f_n\\
g_1&g_2&\cdots&g_n\rix.
$$
Then the Eagon-Northcott complex is:
\begin{equation}\label{eqn_Eagon-Northcott}
0\to (S_{n-2}G)^*\otimes \Lambda^n F\stackrel{e_{n-1}}{\rightarrow} (S_{n-3}G)^*\otimes \Lambda^{n-1} F\stackrel{e_{n-2}}{\rightarrow}\cdots \rightarrow 
(S_{1}G)^*\otimes \Lambda^3 F\stackrel{e_{2}}{\rightarrow}\Lambda^2 F\stackrel{e_{1}}{\rightarrow} A\to A/I\to 0
\end{equation}
Here we let $E_i:=(S_{i-1}G)^*\otimes \Lambda^{i+1}F$
for $i=1, 2,\cdots, n-1$.  $E_i$ has rank 
$$\rk(E_i)=\mat{c}i\\ i-1\rix\cdot \mat{c}n\\i+1\rix.$$
The cotangent complex is given by:
\begin{equation}\label{eqn_cotangent_complex_Lem}
0\to \Omega_{E_{n-1}}\otimes_{A}A/I \stackrel{d_{n-1}}{\rightarrow} \Omega_{E_{n-2}}\otimes_{A}A/I \stackrel{d_{n-2}}{\rightarrow}\cdots \rightarrow 
\Omega_{E_{2}}\otimes_{A}A/I \stackrel{d_{2}}{\rightarrow}\Omega_{E_{1}}\otimes_{A}A/I \stackrel{d_{1}}{\rightarrow} \Omega_{A}\otimes_{A}A/I \to 0
\end{equation}
In the case of calculation of $\sT_X^2$, we can consider the truncation of (\ref{eqn_Eagon-Northcott}).
Let $\Im(e_3)=\ker(e_2)=R$ and we have:
$$0\to R\stackrel{e_{2}}{\rightarrow} E_1\stackrel{e_{1}}{\rightarrow}A\to A/I\to 0$$
From Remark \ref{rmk_cotangent_complex_LS}, 
$$\sT_X^2=\cok\left(\Hom_{\sO_X}(E_1\otimes\sO_X, \sO_X)\rightarrow \Hom_{\sO_X}(R, \sO_X)\right).$$

Let $F_{i,j}:=f_ig_j-f_jg_i$ for $1\leq i<j\leq n$. 
The module $R$ is the relations of $F_{i,j}$'s  which is generated by:
$$
\begin{cases}
R_{i,j,k}=f_i F_{j,k}-f_j F_{i,k}+f_k F_{i,j};\\
S_{i,j,k}=g_i F_{j,k}-g_j F_{i,k}+g_k F_{i,j}
\end{cases}
$$
for $1\leq i<j<k\leq n$. 
It is sufficient to show $f_1\sT_X^2=0$ from \cite[Proposition 2.1.1]{BC}, because changing rows and columns does not change $X$. 
Let us choose $\phi\in \Hom_{\sO_X}(R, \sO_X)$, and let 
$$\phi(R_{i,j,k})=\phi^1_{i,j,k}; \quad   \phi(S_{i,j,k})=\phi^2_{i,j,k}.$$
It is enough to show that $f_1\phi\in \Im(\Hom_{\sO_X}(E_1\otimes\sO_X, \sO_X))$. i.e., show that there exist $\tiny\mat{c} n\\ 2\rix$ elements
$h_{i,j}$ of $\sO_X. (1\leq i<j\leq n)$ such that 
 $$
\begin{cases}
f_1\cdot \phi^1_{i,j,k}=f_i h_{j,k}-f_j h_{i,k}+f_k h_{i,j};\\
f_1\cdot \phi^2_{i,j,k}=g_i h_{j,k}-g_j h_{i,k}+g_k h_{i,j}.
\end{cases}
$$
The existence of such $h_{i,j}$'s are given by the relations among relations of $R_{i,j,k}$ and $S_{i,j,k}$ and there are totally 
$\rk(E_3)=3\cdot \tiny\mat{c}n\\4\rix$ of them. 

Now for the module $\sT_X^3$ for $i\ge 3$, we use similar analysis.  Let us do the case 
$i=3$.  In this case let 
$R=\Im(e_4)=\ker(e_3)$ in (\ref{eqn_Eagon-Northcott}).  We get:
$$0\to R\stackrel{e_{3}}{\rightarrow} E_2 \stackrel{e_{2}}{\rightarrow}E_1\stackrel{e_{1}}{\rightarrow}A\to A/I\to 0$$
and 
$$\sT_X^3=\cok\left(\Hom_{\sO_X}(\Omega_{E_2}\otimes\sO_X, \sO_X)\rightarrow \Hom_{\sO_X}(\Omega_{R}\otimes\sO_X, \sO_X)\right).$$
The module $R=\Im(e_4)=\ker(e_3)$ is generated by 
$$X_{i, j, k, l}, \quad Y_{i, j, k, l}, \quad  Z_{i, j, k, l}$$
for $1\leq i<j<k<l\leq n$.  We show that 
$f_1\phi\in \Im(\Hom_{\sO_X}(E_2\otimes\sO_X, \sO_X))$ for any $\phi\in \Hom_{\sO_X}(R\otimes\sO_X, \sO_X))$.
This needs to show there exist $2\tiny\mat{c}n\\3\rix$ elements $h_{i,j,k}^1,  h_{i,j,k}^2$ such that 
$$
\begin{array}{l}
\phi(X_{i,j,k,l})=f_ih^1_{j,k,l}-f_jh^1_{i,k,l}+f_kh^1_{i,j,l}-f_lh^1_{i,j,k},\\
\phi(Y_{i,j,k,l})=g_ih^1_{j,k,l}-g_jh^1_{i,k,l}+g_kh^1_{i,j,l}-g_lh^1_{i,j,k},\\
\phi(Z_{i,j,k,l})=f_ih^2_{j,k,l}-f_jh^2_{i,k,l}+f_kh^2_{i,j,l}-f_lh^2_{i,j,k}.
\end{array}
$$
These  relations are given by the $\rk(E_4)=4\tiny\mat{c}n\\5\rix$ relations of 
$X_{i, j, k, l}, Y_{i, j, k, l},   Z_{i, j, k, l}$.  For instance, when $n=4$, $f_1\cdot \phi=0$ automatically since there is zero relations of $X_{i, j, k, l}, Y_{i, j, k, l},   Z_{i, j, k, l}$. 
\end{proof}

For the case of minimally elliptic surface singularities, we have
\begin{thm}\label{thm_ellitptic_surface_singularities}
Let $(X,0)$ be a minimally elliptic  surface singularity of embedding dimension $n+1\ge 5$, and 
$f$ is a general hypersurface section.  Then 
\begin{enumerate}
\item $\dim_{\cc}(\sT_X^2/\m\sT^2_X)=\frac{1}{2}(n+1)(n-4)$, where $\m$ is the maximal ideal and 
$\m\sT_X^2=f\sT_X^2$.
\item $\dim_{\cc}(\sT_X^3/\m\sT^3_X)=\frac{1}{6}(n+1)(n-3)(n-4)$.
\item $\dim_{\cc}(\sT_X^4/\m\sT^4_X)=\frac{1}{12}(n+1)(n-4)(n-3)(n-2)$.
\end{enumerate}
\end{thm}
\begin{proof}
The formula (1) is from \cite[Theorem 6.1.1]{BC2}. We calculate  (2) and (3) using 
Theorem \ref{thm_main_thm}:
\begin{align*}
\dim_{\cc}(\sT_X^3/m\sT_X^3)&=\dim_{\cc}(\sT^2_{Y})-\dim_{\cc}(\sT_X^2/\m\sT^2_X)\\
&=\frac{1}{6}n(n+1)(n-4)-\frac{1}{2}(n+1)(n-4)\\
&=\frac{1}{6}(n+1)(n-4)(n-3).
\end{align*}

\begin{align*}
\dim_{\cc}(\sT_X^4/\m\sT_X^4)&=\dim_{\cc}(\sT^3_{Y})-\dim_{\cc}(\sT_X^3/\m\sT_X^3)\\
&=\frac{1}{12}n(n+1)(n-4)(n-3)-\frac{1}{6}(n+1)(n-4)(n-3)\\
&=\frac{1}{12}(n+1)(n-4)(n-3)(n-2).
\end{align*}
\end{proof}

\begin{rmk}
The result in Theorem \ref{thm_ellitptic_surface_singularities}  gives the minimal generators of the modules 
$\sT_X^2, \sT_X^3, \sT_X^4$. 
\end{rmk}

\subsection{Application to semi-log-canonical surfaces}\label{subsec_slc}

Let $X$ be a projective surface over $\cc$.  We say $X$ has semi-log-canonical (s.l.c.) singularities if the following conditions hold.
\begin{enumerate}
\item The surface $X$ is reduced, Cohen-Macaulay, and has only double normal crossing singularities $(xy=0)\subset\cc^3$ away from a finite set of points. 
\item Let $(X^\nu, \Delta^\nu)$ be the normalization of $X$ with the inverse image of the double curve.  Then $(X^{\nu}, \Delta^{\nu})$ has log canonical singularities. 
\item For some $N>0$ the $N$-th reflexive tensor power $\omega_X^{[N]}=(\omega_X^{\otimes N})^{\vee}$ for the dualizing sheaf $\omega_X$ is invertible. 
\end{enumerate}

From \cite[Theorem 4.23, Theorem 4.24]{Kollar-Shepherd-Barron},  s.l.c. surface singularities have the following possibilities: 

\begin{enumerate}
\item  The semi-log-terminal (s.l.t.) surface singularities; 
\item The Gorenstein s.l.c. surfaces;  
\item  the $\zz_2, \zz_3, \zz_4, \zz_6$ quotients of the simple elliptic singularity, and $\zz_2$ quotients of cusps, and  degenerate cusps.
\end{enumerate}
The s.l.t. surface singularities are classified in \cite[Theorem 4.23]{Kollar-Shepherd-Barron}, which are the cyclic quotient singularities in the sense of Brieskorn, or the cyclic quotients of normal crossings or pinch points.
From  \cite[Theorem 4.21 (ii)]{Kollar-Shepherd-Barron}, Gorenstein s.l.c. surface singularities are either semi-canonical singularities or simple  elliptic singularities,  cusp, and degenerate cusp singularities. Here   
semi-canonical surface singularities are either smooth points, normal crossing points, pinch points or DuVal singularities.  

Recall that a stable surface is a s.l.c. surface $S$ such that $\omega_S$ is ample. 
The moduli space of general type stable surfaces has a compactification by adding s.l.c. surfaces in the boundary, see \cite{Kollar-Book}.

Cyclic quotient singularities are rational surface singularities, and cusp singularities are minimally elliptic surface singularities. 
Let $(X,0)$ be a projective general type surface with only $0\in X$ an isolated cyclic quotient singularity. Then $X$ is s.l.c. 
Recall the local to global spectral sequence
$$E_2^{p,q}=H^p(X, \sT_X^q)\Rightarrow T^{p+q}_{X}.$$
The global cohomology $H^i(X, F)=0$ for $i\ge 3$ and any coherent sheaf $F$. 
Since $X$ only has one isolated singularity $p\in X$, and the tangent sheaf $\sT_X^q$ only supports on the singular point 
$0\in X$ for $q\geq 2$,  from the spectral sequence 
$$T^i_{X}=\sT^i_{X}$$
for $i\ge 3$.
Thus we get:
\begin{prop}\label{prop_cyclic_slc}
Let $X$ be a projective s.l.c. general type surface with only finite cyclic quotient singularities.  If there is one singularity germ such that the local embedding dimension $n\ge 4$, then 
the higher obstruction  spaces $T^i_{X}$ do not vanish. 
\end{prop}
\begin{proof}
This is from the above analysis using local to global spectral sequence and the results in Theorem \ref{thm_rational_surface}. 
\end{proof}

Now let $(X,0)$ be a s.l.c. surface with one isolated singularity $0\in X$.  We assume that the singularity $0$ is  one of the following germ singularities:  a simple elliptic singularity, a $\zz_2, \zz_3, \zz_4, \zz_6$-quotient of simple elliptic singularity;   a cusp, and the 
$\zz_2$ quotient of a cusp singularity. 
Then the index of $p\in X$ is the least integer $N$ such that $\omega_X^{[N]}$ is invertible.  We fix an isomorphism 
$\theta: \omega_X^{[N]}\to \sO_X$ and define 
$$Z:=\spec_{\sO_X}(\sO_X\oplus\omega_X\oplus\cdots\oplus\omega_X^{[N-1]})$$
where the multiplication on $\sO_Z$ is defined by the isomorphism $\theta$. 
Then $\pi: Z\to X$ is a cyclic cover of degree $N$ such that  $\pi^{-1}(0)=0$ is a single point in $Z$. This cover  is called the index one cover of $X$.
The surface $Z$ is Gorenstein and $(Z,0)$ is also s.l.c.  
The index one covering Deligne-Mumford stack is $\XX=[Z/\mu_N]$, and \cite{Hacking} proved that the $\qq$-Gorenstein deformation of 
$X$ is equivalent to the deformation of the index one covering Deligne-Mumford stack $\XX=[Z/\mu_N]$. 
Let $\pi: \XX\to X$ be the map to its coarse moduli space. 

Let $\ll^{\bullet}_{\XX}$ be the cotangent complex of $\XX$.  Define
$$\sT_{\QG}^i(X)=\pi_*\sE xt^i(\ll^{\bullet}_{\XX}, \sO_{\XX})$$
and 
$$T_{\QG}^i(X)=\Ext^i(\ll^{\bullet}_{\XX}, \sO_{\XX}).$$
The first two groups $T_{\QG}^1(X)$ and $T_{\QG}^2(X)$ classify the $\qq$-Gorenstein deformations and obstructions of the s.l.c. surface $X$.   The groups $T_{\QG}^i(X)$ for $i\ge 3$
 are higher obstruction spaces for deforming $X$. 
There is also a local to global spectral sequence
$$E_2^{p,q}=H^p(X, \sT_{\QG}^q(X))\Rightarrow T_{\QG}^{p+q}(X).$$
Since $0\in X$ is a cyclic quotient of a simple elliptic singularity or a cusp singularity,  its lift $0\in Z$ must be  a simple elliptic singularity or a cusp singularity since $Z$ is Gorenstein. 
The sheaves $\sT_{\QG}^i(X)$ supports only on the point $0$ when $i\ge 2$.  Thus from the spectral sequence
$$T_{\QG}^i(X)=\sT_{\QG}^i(X)$$
for $i\ge 3$.
Thus we get:
\begin{prop}\label{prop_elliptic_slc}
Let $X$ be a s.l.c. surface with only finite simple elliptic,  cusp  singularities or the cyclic quotients of them.  If the local embedding dimension of the singularity is bigger than $5$,  then 
the higher obstruction  spaces $T_{\QG}^i(X)$ for $i\ge 3$ do not vanish. 
\end{prop}
\begin{proof}
This is from the above analysis  and the results in Theorem \ref{thm_ellitptic_surface_singularities}. 
\end{proof}

\subsection{Application to the perfect obstruction theory}

Let $X$ be a s.l.c. surface, and $\pi: \XX\to X$ be its index one covering Deligne-Mumford stack.   We consider the $\qq$-Gorenstein deformation of $X$ as the deformation of the Deligne-Mumford stack 
$\XX$.   Let us fix two invariants $K^2, \chi\in \zz$, and 
let $M:=\overline{M}_{K^2, \chi}$ be the  moduli functor
$$M: \Sch_{\cc}\to \text{Groupoids}$$
sending 
\begin{equation}\label{eqn_functor1}
T\mapsto  \left\{\text{the groupoid of ~} \sX\to T \left| \begin{array}{l}
  \text{$\bullet \sX\to T$ is a flat family of stable s.l.c. surfaces} \\
  \text{$\bullet$ for each fibre ~} X_t , t\in T, K_{X_t}^2=K^2, \chi(\sO_{X_t})=\chi. \\
     \end{array}  \right\}\right.
     \end{equation}
From \cite{Kollar-Shepherd-Barron}, \cite{Kollar-Book},
the functor $M$ is represented by a proper Deligne-Mumford stack $M:=\overline{M}_{K^2, \chi}$
of finite type over $\cc$ with projective coarse moduli space. 

Let $A$ be a $\cc$-algebra and  $J$ be a finite $A$-module. 
If  $\sX/A$ is a family of $\qq$-Gorenstein deformation of $X$ over $A$ and $\XX/A$ the family of the corresponding  index one  covering Deligne-Mumford  stacks, then  we define
$T_{\QG}^i(\sX/A,  J):=\Ext^i(\ll^{\bullet}_{\XX/A},  \sO_{\XX}\otimes_{A} J)$. 
From \cite{Hacking}, the set of isomorphism classes of $\qq$-Gorenstein deformations of $\sX_0/A_0$ over $A_0+J$ is canonically isomorphic to $T_{\QG}^1(\sX_0/A_0,  J)$.  Let $A^\prime\to A\to A_0$ be extensions such that the kernel of $A^\prime\to A$ is $J$, then  there exists a canonical element 
$\Ob(\sX/A, J)\in T_{\QG}^2(\sX/A,  J)$ called the obstruction class such that it  vanishes if and only if there exists a $\qq$-Gorenstein deformation $\sX^\prime/A^\prime$ over $A^\prime$. 

We define a new functor:
$$M^{\ind}:=\overline{M}^{\ind}_{K^2, \chi}: \Sch_{\cc}\to \text{Groupoids}$$
which is given by:
\begin{equation}\label{eqn_functor2}
T\mapsto  \left\{\text{the groupoid of ~} \XX\to T \left| \begin{array}{l}
  \text{$\bullet \XX\to T$ is the index one covering Deligne-Mumford stacks}\\
  \text{of a flat family of stable s.l.c. surfaces} \\
  \text{$\bullet$ the coarse moduli space $\sX$ of $\XX$ satisfies the conditions in} \\
  \text{(\ref{eqn_functor1}).} \\
     \end{array}  \right\}\right.
     \end{equation}
     
We have:

\begin{thm}\label{thm_index_one_covering}(\cite[Theorem 4.1]{Jiang_2021})
The functor $M^{\ind}=\overline{M}^{\ind}_{K^2, \chi}$ is representable and has  unramified diagonal, therefore $M^{\ind}$ is a fine Deligne-Mumford stack  with projective coarse moduli space.  Furthermore there is a canonical morphism 
$$f: \overline{M}^{\ind}_{K^2, \chi}\to M=\overline{M}_{K^2, \chi}$$
which is an \'etale morphism. 
The two Deligne-Mumford stacks have  same projective coarse moduli space. 
\end{thm}

The index one covering Deligne-Mumford stack $M^{\ind}$ is a fine moduli stack, therefore there exists a universal family 
$p^{\ind}: \sM^{\ind}\rightarrow M^{\ind}$, which is a projective, flat and relative Gorenstein morphism between Deligne-Mumford stacks. 
Let 
$$E^{\bullet}_{M^{\ind}}:=Rp^{\ind}_{*}(\ll^{\bullet}_{\sM^{\ind}/M^{\ind}}\otimes \omega_{\sM^{\ind}/M^{\ind}})[-1]$$
where $\ll^{\bullet}_{\sM^{\ind}/M^{\ind}}$ is the relative cotangent complex of $p^{\ind}$, and $\omega_{\sM^{\ind}/M^{\ind}}$ is the relative dualizing sheaf 
of $p^{\ind}$ which is a line bundle.   We have  the moduli of projective Deligne-Mumford stacks satisfying the condition in  \cite[Proposition 6.1]{BF}.  
Thus the Kodaira-Spencer map $\ll^{\bullet}_{\sM^{\ind}/M^{\ind}}\to (p^{\ind})^{*}\ll_{M^{\ind}}^{\bullet}[1]$ induces an obstruction theory 
\begin{equation}\label{eqn_OT}
\phi^{\ind}: E^{\bullet}_{M^{\ind}}\to \ll_{M^{\ind}}^{\bullet}
\end{equation}
on $M^{\ind}$, see \cite[Theorem 4.5]{Jiang_2021}.

The calculation in Proposition \ref{prop_elliptic_slc} implies the following result:

\begin{thm}\label{thm_M_BF_LT}
Let $M:=\overline{M}_{K^2, \chi}$ be the main  connected component of the moduli stack of s.l.c. surfaces containing smooth surfaces with $K_S^2=K^2, \chi(\sO_S)=\chi$, and let $M^{\ind}\to M$ be the index one covering Deligne-Mumford stack. 

If there exist s.l.c. surfaces $X$ on the boundary of 
$M$ such that $X$ contains the following surface singularities: 
\begin{enumerate}
\item simple elliptic singularities $(X,0)$ with embedded dimension $\ge 6$, or  the $\zz_2, \zz_3, \zz_4, \zz_6$-quotient of simple elliptic   singularities $(X,0)$ with local embedded dimension $\ge 6$;
\item cusp germ singularities $(X,0)$ with local embedded dimension $\ge 6$, or the $\zz_2$-quotient cusp  singularities $(X,0)$ with local embedded dimension $\ge 6$.
\end{enumerate} 
then the higher obstruction spaces $T_{\QG}^i(X)\neq 0$ for $i\ge 3$. Thus there is no Li-Tian and Behrend-Fantechi style perfect obstruction theory for the moduli stack $M$, and there is  no such style virtual fundamental class on 
$M$. 
\end{thm}
\begin{proof}
For a s.l.c. surface $X$, and its index one covering Deligne-Mumford stack $\pi: \XX\to X$, the restriction of the obstruction complex $E^{\bullet}_{M^{\ind}}$ to $X$ calculates the groups 
$T_{\QG}^i(X)$ for $i\ge 1$.   From Proposition  \ref{prop_elliptic_slc}, for such s.l.c. surface singularities, $T_{\QG}^i(X)\neq 0$ for $i\ge 3$, i.e., the higher obstruction spaces do not vanish.  
Therefore the complex $E^{\bullet}_{M^{\ind}}$ is not perfect with amplitude contained in $[-1,0]$.    The obstruction theory (\ref{eqn_OT}) is not perfect in the sense of Behrend-Fantechi and Li-Tian, which implies there is no 
Behrend-Fantechi and Li-Tian style virtual fundamental class on $M^{\ind}$. 
\end{proof}

%%%---------------------------------------------------------------------------------

% ------------------------------------------------------------------------
\end{document}